\documentclass[10pt]{article}

\usepackage{amssymb} 
\usepackage{latexsym} 
\usepackage{amsmath}
\usepackage{mathtools} 
\usepackage{thmtools}
\usepackage{tabls} 
\usepackage{graphicx}
\usepackage{todonotes}
\usepackage{overpic}
\usepackage{subcaption}
\usepackage{dutchcal}
\usepackage[english]{babel}
\usepackage{ntheorem}
\usepackage{a4wide}
\usepackage[utf8]{inputenc}
\usepackage{color}
\usepackage{epsfig}
\usepackage{graphicx}
\usepackage[all]{xy}
\input xy
\usepackage[mathscr]{eucal} %use \mathsrc{X} to type eucal X
\usepackage{enumerate}
\usepackage{enumitem}
\usepackage{accents}
\newtheorem{theorem}{Theorem}
\newtheorem{definition}{Definition}

\newtheorem{lemma}{Lemma}
\newtheorem{example}{Example}
\newtheorem*{property}{Chain Property}
\newtheorem{corollary}[theorem]{Corollary}

\newenvironment{proof}{\noindent{\bf Proof.}}%
{\hspace*{\fill}$\Box$\par\vspace{4mm}}

\def\cadre{$$\vcenter\bgroup\advance\hsize by -2em\noindent
	\refstepcounter{equation}(\theequation)~\ignorespaces}
\makeatletter
\def\endcadre{\egroup\eqno$$\global\@ignoretrue}

\newcommand{\mybreak} {\par\vspace{2mm}\noindent}

\makeatletter
\def\imod#1{\allowbreak\mkern10mu({\operator@font mod}\,\,#1)}
\makeatother

\newcommand{\comment}[1]{}

\newcommand{\pr} {{\rm Pr}}
\newcommand{\pra}[1] {\pr\left\{#1\right\}}
\newcommand{\A} {\mathbb{A}}
\newcommand{\E} {\mathbb{E}}

\newcommand{\N} {\mathbb{N}}
\newcommand{\G} {\mathbb{G}}
\newcommand{\R} {\mathbb{R}}
\newcommand{\Gc} {\mathcal{G}}
\newcommand{\F} {\mathcal{F}}

\newcommand{\Ea}[1] {\E\left(#1\right)}
\setlist[itemize]{labelindent=\parindent, leftmargin=0.4cm}

\title{Functions that are uniquely maximized by sparse quasi-star graphs, and uniquely minimized by quasi-complete graphs} 
\author{AA}
\date{}
\author{Nicola Apollonio\footnote{Istituto per le Applicazioni del
		Calcolo, M. Picone, Via dei Taurini 19, 00185 Roma, Italy.
		\texttt{nicola.apollonio@cnr.it}}
}

\begin{document}
	\pagestyle{plain}
	\renewcommand\thmcontinues[1]{Continued}
	\maketitle	
\begin{abstract}
We show that for a certain class of convex functions $f$, including the exponential functions $x\mapsto e^{\lambda x}$ with $\lambda>0$ a real number, and all the powers $x\mapsto x^\beta$, $x\geq 0$ and $\beta\geq 2$ a real number, with a unique small exception, if $(d_1,\ldots,d_n)$ ranges over the degree sequences of graphs with $n$ vertices and $m$ edges and $m\leq n-1$, then the maximum of $\sum_i f(d_i)$ is uniquely attained by the degree sequence of a quasi-star graph, namely, a graph consisting of a star plus possibly additional isolated vertices. This result significantly extends a similar result in [D.~Ismailescu, D.~Stefanica, Minimizer graphs for a class of extremal problems, J.~Graph Theory,~39~(4)~(2002)]. Dually, we show that for a certain class of concave functions $g$, including the negative exponential functions $x\mapsto 1-e^{-\lambda x}$ with $\lambda>\ln(2)$ a real number, all the powers $x\mapsto x^\alpha$, $x\geq 0$ and $0<\alpha\leq \frac{1}{2}$ a real number, and the function $x\mapsto \frac{x}{x+1}$ for $x\geq 0$, if $(d_1,\ldots,d_n)$ ranges over the degree sequences of graphs with $n$ vertices and $m$ edges, then the minimum of $\sum_i g(d_i)$ is uniquely attained by the degree sequence of a quasi-complete graph, i.e., a graph consisting of a complete graph plus possibly an additional vertex connected to some but not all vertices of the complete graph, plus possibly isolated vertices. This result extends a similar result in the same paper.
	\end{abstract}
	\mybreak
\section{Introduction}
Let $\G(n,m)$ be the set of graphs with $n$ vertices and $m$ edges. Also, let $\Sigma$ be a real-valued function whose domain contains $\N^n$. A graph $G\in \G(n,m)$ is a \emph{maximizer} of $\Sigma$ over $\G(n,m)$ if the degree sequence $d$ of $G$ achieves the maximum possible value of $\Sigma$ among all the degree sequences of members of $\G(n,m)$. If $G$ is a maximizer, then so is any graph isomorphic to $G$. In general, however, non-isomorphic graphs might have the same degree sequence. We say that $G$ \emph{uniquely maximizes} $\Sigma$ if, up to isomorphism, $G$  is the unique graph whose degree sequence maximizes $\Sigma$. In this paper we are interested in functions $\Sigma$ of the form $(x_1,\ldots,x_n)\mapsto \sum_i^nf(x_i)$ where $f$ is a real-valued function whose domain contains the natural numbers and $f$ satisfies the following ``integer strict convexity'' condition
\begin{equation}\label{eq:Nstrict}
f(k+2)-f(k+1)> f(k+1)-f(k),\quad \forall k\in \N. 
\end{equation}
Every strictly convex function $f$ on a convex set $D$ satisfies \eqref{eq:Nstrict} provided that $D\supseteq \N$. In particular, so does $\exp_\lambda: \R\rightarrow \R$, $x\mapsto e^{\lambda x}$ for $\lambda$ a positive real number and $p_\beta:\R_+\rightarrow \R$, $x\mapsto x^\beta$, for $\beta$ a real number such that $\beta>1$. For the functions $p_\beta$ with $\beta \in \R$, the notation $\Sigma_{p_\beta}$ is usually abridged to $\Sigma_\beta$. In particular, if $\beta\in \N$, then $\Sigma_\beta$ is the power-sum of degree $\beta$. The most celebrated and studied instance of $\Sigma_\beta$ is undoubtedly $\Sigma_2$, namely the sum of the squares of the degrees of a graph, which has been investigated in literally dozens of papers by several authors from different viewpoints and it is still a central topic in the mathematical chemistry literature where $\Sigma_2$ appears under the name of Zagreb index~\cite{dasetalsurvey}. We refer to \cite{alhswedeetal,nikiforov} for a brief history on the exact maximization of $\Sigma_2$. Borrowing the terminology of \cite{abregoetal, alhswedeetal, chineu}, the result below shows that $\Sigma_2$ is maximized exactly. 
\begin{theorem}[\cite{abregoetal,alhswedeetal,ifive}]\label{thm:sigma_2th}
If $m$ and $n$ are given integers such that $n>0$ and $0\leq m\leq {n \choose 2}$, then 
the \emph{quasi-complete} graph $QK(m,n)$ or the \emph{quasi-star} graph $QS(n,m)$ maximizes $\Sigma_2$ over $\G(n,m)$ where  
\begin{itemize}
	\setlength\itemsep{0.02cm}
	\setlength\parskip{0pt}
\item[--] $QK(n,m)$ is the graph on $\{1,\ldots,n\}$ such that, for integers $q$ and $r$ uniquely determined by the relations $m={q \choose 2}+r$ and $0\leq r\leq q-1$, the vertices $1,\ldots,q$ induce a complete graph, vertex $q+1$ is adjacent to vertices $1,\ldots, r$, and vertices $q+2,\ldots,n$ are isolated;          
\item[--] $QS(n,m)$ is the complement of $QK(n,{n \choose 2}-m)$.   
\end{itemize}
\end{theorem}
In \cite{sterbo}, Peled, Petreschi and Sterbini refined Theorem~\ref{thm:sigma_2th} by showing that there are six isomorphism classes of graphs that maximize $\Sigma_2$. Note that if $m\leq n-1$, then $QS(n,m)$ is isomorphic to a star on $m$ leaves plus $n-m-1$ isolated vertices. In the special case of very sparse graphs, namely, when $m\leq n-2$, Ismailescu and Stefanica \cite{IsmaStef}, significantly refined Theorem~\ref{thm:sigma_2th} by showing that, with a single small exception, $QS(n,m)$ uniquely maximizes not only $\Sigma_2$ but also $\Sigma_\nu$ for every integer $\nu\geq 2$. In this paper we give a significant extension of the latter result by showing that, with a single small exception, for all $m\leq n-1$, the quasi-star graph $QS(n,m)$ uniquely maximizes $\Sigma_f$ for every function $f$ in a class $\F$ that contains, for instance, all the powers $p_\beta$ for $\beta\geq 2$ a real number, all the exponentials $\exp_\lambda$ for $\lambda>0$ a real number, and any linear combination of such functions with non-negative coefficients at least one of which is positive. Remark that not every function that satisfies \eqref{eq:Nstrict} is maximized over $\G(n,m)$ by $QS(n,m)$ when $m\leq n-1$. For instance, in Theorem~\ref{ex:1} we show that for $q\geq 4$ an integer, there exists a real number $\epsilon(q)$, such that $0<\epsilon(q)<1$ and $\Sigma_{1+\epsilon(q)}$ is not maximized by $QS(n,m)$ for $m={q \choose 2}$ and $n=m+1$. 

\noindent In the same paper \cite{IsmaStef}, Ismailescu and Stefanica, extending earlier work of Linial and  Rozenman~\cite{LinialRoze}, also gave a dual result, namely, that $\Sigma_\alpha$  is uniquely minimized by the quasi-complete graph $QK(n,m)$ over $\G(n,m)$ for all real numbers $\alpha\in (0,\frac{1}{2}]$. We extend the latter result by showing that the quasi-complete graph $QK(n,m)$ uniquely maximizes $\Sigma_g$ for any function $g$ in a class of functions containing, for instance, besides all the powers $p_\alpha$ for a real number $\alpha\in (0,\frac{1}{2}]$ also the negative exponentials $1-\exp_{-\lambda}$ for $\lambda>\ln{}2$ a real number, and any linear combination of such functions with non-negative coefficients at least one of which is positive. This result is particularly interesting when compared to Theorem~\ref{thm:sigma_2th} and its refinement in \cite{sterbo} because $\Sigma_2$ has multiple maximizers while, for instance,  $-\Sigma_{\frac{1}{2}}$ has a unique maximizer. Moreover, as shown in \cite{LinialRoze} for the function $-\Sigma_{\frac{1}{2}}$, and in \cite{sterbo} for the function $\Sigma_2$, both functions are maximized by the \emph{threshold graphs} in $\G(n,m)$, namely those graphs that satisfy the following property.
\begin{property} A graph $G$ has the \emph{chain property} if for any three distinct vertices $i$, $j$ and $k$ it holds that if $d_i\geq d_j$ and $j$ and $k$ are adjacent, then $i$ and $k$ are also adjacent.  
\end{property}
Threshold graphs are characterized in several other equivalent ways \cite{chvHam,golumbic,mahapele}. For our purposes, however, the Chain Property is the only property we need to show that $\Sigma_f$ is maximized by the threshold graphs in $\G(n,m)$ for every $f$ that satisfies \eqref{eq:Nstrict}. The proof of this fact is remarkably simple and it is just an abstraction of the corresponding proofs given in \cite{LinialRoze} and \cite{sterbo} in the special cases. 
\mybreak

\subsection{Preliminaries} Let us give some notation that will be used throughout the rest of the paper. For a positive integer $n$, $[n]$ is the set of the first $n$ positive integers. Since the property of being a graph maximizer of $\Sigma$ is invariant under graph isomorphism, the vertex set $V(G)$ of any graph $G$ considered in this paper is $[n]$. Accordingly, the degree sequence of $G$ is the sequence $d=(d_1,\ldots,d_n)$ where $d_i$ is the degree of vertex $i$, $i\in [n]$. To say that two vertices $i$ and $j$ of $G$ are adjacent we write $i\sim j$. The \emph{conjugate sequence} of the degree sequence $d=(d_1,\ldots,d_n)$ is the sequence $d^*=(d^*_1,\ldots,d^*_n)$ where $d_i^*$ is defined by the rule $d_i^*=\#\{j\in[n] \ | \ d_j\geq i\}$. For two graphs $G$ and $H$, the graph $G\oplus H$ is isomorphic to the graph theoretic union of two vertex-disjoint copies of $G$ and $H$, respectively. The graph $\overline{G}$ is the complement of $G$.
\section{Threshold graphs maximize strictly convex functions}
As mentioned in the previous section, \cite{sterbo} proves that the Chain Property is a necessary condition for maximizing $\Sigma_2$ over $\G(n,m)$ and \cite{LinialRoze} proves that the Chain Property is a necessary condition for minimizing $\Sigma_\alpha$ for $\alpha\in (0,\frac{1}{2}]$. Here we generalize their results to all the functions $\Sigma_f$ where $f$  satisfies \eqref{eq:Nstrict}. In particular, the result holds for strictly convex (or strictly concave) functions $f$ whose domain contains the natural numbers. Recall that a strictly convex function is a real-valued function $f$ defined on an interval $I$ of $\R$, such that 
$$f\left(\lambda x+(1-\lambda)y\right)<\lambda f(x)+(1-\lambda)f(y)$$
for every $x,\,y\in I$, $x\not=y$, and for every $\lambda\in (0,1)$. 
A useful geometrical characterization of strict convexity is the following. 
\begin{lemma}\label{lem:slopecon}
A real-valued function $f$ defined on an interval $I$ of $\R$ is strictly convex if and only if for any three $x_1,\,x_2,\,x_3\in I$ such that $x_1<x_2<x_3$ the following three-slopes condition holds 
$$\frac{f(x_2)-f(x_1)}{x_2-x_1}< \frac{f(x_3)-f(x_1)}{x_3-x_1}< \frac{f(x_3)-f(x_2)}{x_3-x_2}.$$
\end{lemma}
By allowing weak inequalities in all the above relations, one recovers the usual notion of convexity. It is clear that if a strictly convex function $f$ is defined for the natural numbers, then $f$ satisfies \eqref{eq:Nstrict}, since such a condition is a special case of the three-slopes condition given in the lemma. The converse is also true in the following sense.
\begin{lemma}[Slopes condition]\label{lem:slope}
If $f: \N\rightarrow \R$ satisfies \eqref{eq:Nstrict}, then $f$ satisfies the three-slopes condition in Lemma \ref{lem:slopecon} for any three $x_1,\,x_2,\,x_3\in \N$ such that $x_1<x_2<x_3$.
\end{lemma}
\begin{proof}
Let $x_1=a$, $x_2=a+b$ and $x_3=a+b+c$, where $a,\,b,\,c$ are positive integers. Let $\nabla_i=f(i)-f(i-1)$, $i=a+1,\ldots a+b+c$ and observe that $f(a+b+c)-f(a)=\sum_{i=a+1}^{a+b+c}\nabla_i$. The sequence $\left(\nabla_i\right)$ is strictly increasing in $i$ for $i\in \{a+1,\ldots,a+b+c\}$. This sequence has $b+c$ elements the first $b$ of which are smaller than the last $c$ ones. Therefore
$$\frac{\sum_{i=a+1}^{a+b}\nabla_i}{b}<\frac{\sum_{i=a+1}^{a+b+c}\nabla_i}{b+c}<\frac{\sum_{i=a+b+1}^{a+b+c}\nabla_i}{c}$$
because the average of $b+c$ elements is always less than the average of the largest $c$ elements and always greater than the average of the smallest $b$ elements.      
\end{proof}
Not every function satisfying \eqref{eq:Nstrict} is the restriction to $\N$ of a strictly convex function whose domain contains $\N$. For instance,
let $\ell_{k+1}:\R\rightarrow \R$ be the function whose graph in $\R^2$ is the straight line through the points $(k,f(k))$ and $(k+1,f(k+1))$. Then the function $\ell:\R\rightarrow \R$ defined by $\ell=\max_{k\in \N}\ell_{k+1}$ is a piecewise linear convex function. It satisfies \eqref{eq:Nstrict} but is not strictly convex.
\mybreak
We show the necessity of the Chain Property when maximizing $\Sigma_f$ for any $f$ that satisfies \eqref{eq:Nstrict}.
\begin{theorem}\label{thm:opt}
	Let $f$ be any function satisfying \eqref{eq:Nstrict}. If $G$ maximizes $\Sigma_f$ over $\G(n,m)$, then $G$ is a threshold graph.
\end{theorem}	
\begin{proof}
Let $G$ be an arbitrary graph in $\G(n,m)$ and let $d$ be its degree sequence. If $G$ is not a threshold graph, then $G$ violates the Chain Property. Hence there exist $i,\,j,\,k\in [n]$ such that $d_i\geq d_j$, and $i\not\sim k$ while $j\sim k$. The graph $G'=([n], E(G)\Delta \{ik,jk\})$ has the same size as $G$ and degree sequence $d'=(d_1,\ldots,d_i+1,\ldots,d_j-1,\ldots,d_n)$. Since $d_i+1>d_j$,  \eqref{eq:Nstrict} implies that $$f(d')-f(d)=\left(f(d_i+1)-f(d_i)\right)-\left(f(d_j)-f(d_j-1)\right)>0.$$  We conclude that $G$ cannot be optimal. 
\end{proof}
The concave counterpart follows by reversing the inequality in the theorem.
\begin{theorem}\label{thm:opt_1}
Let $g$ be any function such that $-g$ satisfies \eqref{eq:Nstrict}. If $G$ minimizes $\Sigma_g$ over $\G(n,m)$, then $G$ is a threshold graph.
\end{theorem}
Actually, the result also holds for the class of \emph{strictly Schur-convex} functions  $\Sigma:\R^n\rightarrow \R$ (see \cite{MarOlk,stepniak} for definitions). As observed in \cite{abregoetal}, the reason is that the degree sequences of threshold graphs are maximal with respect to the majorization order among all graphical sequences. The function $\Sigma_f$ is Shur-convex whenever $f$ is convex, and it is strictly Shur-convex when $f$ is strictly convex. This shows that $\Sigma$ may not be maximized by threshold graphs alone if $\Sigma$ is Shur-convex but not strictly Shur-convex. In light of the last two theorems, maximizers graphs of $\Sigma_f$ for $f$ a strictly convex function, and minimizers of $\Sigma_g$ for $g$ a strictly concave function must be sought among threshold graphs.   
\begin{corollary}\label{cor:unique_max}
Let $f$ satisfy \eqref{eq:Nstrict}.~If $\Sigma_f$ has a unique maximizer $d$ over the degree sequences of graphs in $\G(n,m)$, then $\Sigma_f$ is uniquely maximized by a threshold graph.
\end{corollary}
\begin{proof}
The result follows directly from Theorem \ref{thm:opt} by the fact (proved in \cite{bruno}) that if $d$ is the degree sequence of a threshold graph, then there exists a unique (up to isomorphism) graph with degree sequence $d$.  
\end{proof}
In the next lemma we collect the only three other properties of threshold graphs needed for our purposes, the second of which was given in \cite{IsmaStef}.
\begin{lemma}\label{lem:twoproperties}
Let $G$ be a threshold graph with $n$ vertices, degree sequence $d$, maximum degree $\Delta$ and minimum degree $\delta$. Then
\begin{enumerate}[label={\rm (\roman*)}, ref=\roman*]
	\item\label{com:i} each vertex of degree $\Delta$ is adjacent to every non-isolated vertex of $G$; 
	\item\label{com:ii} if $\delta>0$, then $\Delta=n-1$ and the neighbors of every vertex of degree $\delta$ have degree $n-1$;
	\item\label{com:iii} if $d_2^*>1$, then $G$ contains a triangle and, consequently, $d_2^*\geq 3$.  
\end{enumerate}
\end{lemma}
\begin{proof}
The first statement is a direct consequence of the Chain Property while the second statement (still a direct consequence of the Chain Property) is proved in \cite{IsmaStef}. Let us prove the third statement. If $d_2^*>1$, then $G$ has at least two vertices $u$ and $v$ whose degree is at least 2. We can assume that $u$ has degree $d_1$ and $v$ has degree $d_2$. By the Chain Property, $u$ and $v$ have a common neighbor $w$. Hence $u$, $v$ and $w$ induce a triangle. Since these three vertices have at least degree two, it follows that $d_2^*=\#\{j \ |\ d_j\geq 2\}\geq 3$.  
\end{proof}

\section{Functions uniquely maximized by sparse quasi-star graphs}
In this section we give a sufficient condition for a function satisfying \eqref{eq:Nstrict} to be uniquely maximized by $QS(n,m)$ over $\G(n,m)$ when $m$ is such that $1\leq m\leq n-1$, namely, in the very sparse case. Such a condition yields a class $\F$ of functions which is significantly larger than the class of elementary powers $\{p_\nu: \R_+\rightarrow\R,\,x\mapsto x^\nu \ |\ \nu\in \N,\, \nu\geq 2\}$ devised in \cite{IsmaStef}. Without loss of generality, all of the functions $f$ dealt with throughout the rest of the paper are assumed to be \emph{centered} i.e., $f(0)=0$, because  $\Sigma_{f}$ and $\Sigma_h$ have the same maximizers whenever $f$ and $h$ differ by an additive constant.    
\begin{definition}
The class $\F$ consists of all centered functions $f$ that satisfy both \eqref{eq:Nstrict} and  
\begin{equation}\label{eq:forconvex}
(k-2)f(k)\geq k\left(f(k-1)-f(1)\right),\quad k\geq 3,\,k\in\N.
\end{equation} 
\end{definition} 
Note that every real-valued function $f$ whose domain contains $\N$ satisfies inequality $\eqref{eq:forconvex}$ as an equality for $k=0,1,2$. So we can remove the condition $k\geq 3$ in the definition. Remark that inequality \eqref{eq:forconvex} and the strict convexity of $f$ are independent conditions: it is readily checked that the function $p_{1/2}$ defined by $x\mapsto \sqrt{x}$ satisfies \eqref{eq:forconvex} but violates \eqref{eq:Nstrict} because $p_{1/2}$ is strictly concave, while the function $p_{1+\epsilon}$, where $\epsilon$ is a sufficiently small positive real number, is strictly convex but violates \eqref{eq:forconvex} (see Theorem~\ref{ex:1}).    
\mybreak
To prove the main result of this section, we need the following inequality due to Chebyshev. It can be found in the monograph~\cite{mitri}.
\begin{lemma}[Chebyshev's Algebraic Inequality]\label{lem:cheb}
Let $(a_1,a_2,\ldots,a_n)$ and $(b_1,b_2\ldots,b_n)$ be two sequences of real numbers. If
$$a_1\leq a_2\leq\cdots\leq a_n\quad\text{\rm and}\quad b_1\geq b_2\geq\cdots\geq b_n,$$ 
then
$$\sum_ia_ib_i\leq \frac{1}{n}\left(\sum_ia_i\right)\left(\sum_ib_i\right).$$
Equality in the inequality above holds if and only if at least one of the two sequences is constant.   
\end{lemma}
 
\begin{theorem}\label{thm:main}
	If $f\in \F$, then $QS(n,m)$ maximizes $\Sigma_f$ over $\G(n,m)$ for every two integers $n$ and $m$ such that $0\leq m\leq n-1$. More precisely,
	\begin{enumerate}[label={\rm (\roman*)}, ref=\roman*]
		\item\label{com:i_inth} if $m\not=3$ or $f(3)>3(f(2)-f(1))$, then $QS(n,m)$ uniquely maximizes $\Sigma_f$ over $\G(n,m)$;
		\item\label{com:ii_inth} if $f(3)=3(f(2)-f(1))$, then $K_3\oplus \overline{K}_{n-3}$ also maximizes $\Sigma_f$ over $\G(n,3)$ for any $n\geq 4$.
	\end{enumerate}
\end{theorem}
\begin{proof}
We have already observed that $QS(n,m)\cong K_{1,m}\oplus\overline{K}_{n-m-1}$ when $m\leq n-1$. Let $f$ be an arbitrary function in $\F$. By Theorem \ref{thm:opt}, the maximum of $\Sigma_f$ is attained by the degree sequence of a threshold graph. If $m\leq 2$, up to isomorphism, the only threshold graphs are $\overline{K}_{n}$, $K_2\oplus\overline{K}_{n-2}$, and $K_{1,2}\oplus\overline{K}_{n-3}$. Since all of these graphs are quasi-star graphs, the thesis holds true for $m\leq 2$ and we can assume that $m\geq3$. Let $d=(d_1,\ldots,d_n)$ be the degree sequence of a threshold graph in $\G(n,m)$ with $m\geq 3$. Assume without loss of generality that $d_1\geq d_2\ldots\geq d_n$. Then $d_1$ is the maximum degree of $G$. By Lemma~\ref{lem:twoproperties}.\eqref{com:i}, it holds that $d_1\geq 2$.  As in Lemma~\ref{lem:slopecon}, note that for every $k\in N$ 
$$f(k)=\left(f(k)-f(k-1)\right)+\left(f(k-1)-f(k-2)\right)+\cdots+\left(f(1)-f(0)\right).$$
Since $f(0)=0$, after setting $\nabla_k=f(k)-f(k-1)$ for $k\geq 1$, we can write
$$f(k)=\sum_{l=1}^k\nabla_l.$$
By definition, the value $\Sigma_f(d)$ of $\Sigma_f$ in $d$ is $\sum_i^nf(d_i)$. Therefore
$$\Sigma_f(d)=\sum_{i=1}^n\sum_{l=1}^{d_i}\nabla_l=\sum_{i=1}^{d_1}\nabla_id_i^*$$ 
because $\nabla_i$ is added exactly $\#\{j \ |\ j\geq i\}$ times and $d_i^*=0$ if $i>d_1$. In particular, it follows that
\begin{equation}\label{eq:int_1}
\Sigma_f(d)=f(1)d_1^*+\sum_{i=2}^{d_1}\nabla_id_i^*.
\end{equation}     
Since $d$ is the degree sequence of a threshold graph, Lemma~\ref{lem:twoproperties}.\eqref{com:i} implies that $d_1^*=d_1+1$. Furthermore, the $\nabla_i$'s and the $d_i^*$'s satisfy 
\begin{equation}\label{eq:twosequences}
\begin{split}
\nabla_{i-1}<\nabla_{i},\quad &i=2,\ldots d_1\quad\text{\rm (by \eqref{eq:Nstrict})}\\
d_{i-1}^*\geq d_i^*,\quad &i=2,\ldots d_1\quad\text{\rm (by definition)}.\\
\end{split}
\end{equation}
Therefore, by Chebyshev's inequality, for every threshold graph whose degree sequence is $d$, it holds that
\[
\begin{split}
\sum_{i=2}^{d_1}\nabla_id_i^*&\leq \frac{1}{d_1-1}\left(\sum_{i=2}\nabla_i\right)\left(\sum_{i=2}d_i^*\right)=\\
&=\frac{1}{d_1-1}\left(f(d_1)-f(1)\right)\left(2m-d_1-1\right)
\end{split}
\]
because $2m=\sum_{i=1}^nd_i^*=d_1^*+(2m-d_1^*)$ and $d_1\geq 2$. Plugging the above inequality in \eqref{eq:int_1}, yields 
\begin{equation}\label{eq:int_2}
\Sigma_f(d)\leq f(1)(d_1+1)+\frac{2m-d_1-1}{d_1-1}\left(f(d_1)-f(1)\right),\quad d_1\geq 2
\end{equation}
which holds for any threshold graph whose degree sequence is $d$. If $\tilde{d}$ is the degree sequence of $QS(n,m)$, then $\tilde{d}^{\,*}=(m+1,1,\ldots,1,0,\ldots,0)$, where the number of zeroes in $\tilde{d}^{\,*}$ is $n-m-1$. Therefore, $\tilde{d}$ satisfies \eqref{eq:int_2} as an equality and $\Sigma(\tilde{d})=mf(1)+f(m)$. Consequently, to prove the theorem, it suffices to show that 
\begin{equation}\label{eq:int_3}
mf(1)+f(m)\geq f(1)(k+1)+\frac{2m-k-1}{k-1}\left(f(k)-f(1)\right),\quad 2\leq k\leq m-1
\end{equation}
where we replaced $d_1$ by $k$ in the right-hand side of \eqref{eq:int_2}. Now it is just a matter of checking that condition \eqref{eq:forconvex} in the definition of $\F$ and the integer strict convexity condition \eqref{eq:Nstrict} imply \eqref{eq:int_3}. 

\noindent First assume that $m-k=1$. In this case \eqref{eq:int_3} is equivalent to
$$(m-2)f(m)\geq m(f(m-1)-f(m)),\quad m\geq 3$$
which is precisely \eqref{eq:forconvex} which holds true by hypothesis. Therefore, when $m-k=1$, \eqref{eq:int_3} holds true. 

\noindent Now assume that $m-k>1$. Recall that \eqref{eq:forconvex} holds for $k\geq 3$ so that, after shifting $k$, it can be rewritten as
$$\frac{2\left(f(k)-f(1)\right)}{k-1}\leq \left(f(k+1)-f(k)\right)+f(1),\quad k\geq 2.$$
By multiplying both sides of the above inequality by $(m-k)$, we get
\[
\begin{split}
\frac{2(m-k)}{k-1}\left(f(k)-f(1)\right)&\leq (m-k)\left(f(k+1)-f(k)\right)+(m-k)f(1)\\ 
&< \left(f(m)-f(k)\right)+(m-k)f(1),\quad k\geq 2,
\end{split}
\]
where in the last inequality we used the slope condition $f(k+1)-f(k)<\frac{f(m)-f(k)}{m-k}$ given in Lemma~\ref{lem:slopecon}. Using that $1+\frac{2(m-k)}{k-1}=\frac{2m-k-1}{k-1}$, simple manipulations in the inequality above, yield
$$f(1)(k+1)+\frac{2m-k-1}{k-1}\left(f(k)-f(1)\right)<mf(1)+f(m),\quad k\geq 2$$
which is precisely inequality \eqref{eq:int_3} we wanted to prove and which is a strict inequality for every $k$ such that  $2\leq k<m-1$.  Therefore, if $m-k>1$, not only does $QS(n,m)$ maximize $\Sigma_f$, but, according to Theorem~\ref{thm:opt_1}, it also uniquely maximizes $\Sigma_f$. In any case the first part of the theorem is proved. To complete the proof, it suffices to show that equality holds in \eqref{eq:int_3} if and only if  
\begin{enumerate}[itemsep=1pt,parsep=2pt,topsep=3pt,label={\rm (\alph*)}, ref=\alph*]
	\item\label{com:a_int} $m=3$,      
	\item\label{com:b_int} $f(3)=3(f(2)-f(1))$ and $K_3\oplus \overline{K}_{n-3}$ also maximizes 
	$\Sigma_f$ over $\G(n,3)$.   
\end{enumerate}
\noindent By the first part of the proof, equality in \eqref{eq:int_3} can hold if and only if $m\geq 3$ and $m=k+1=d_1+1$. Observe that \eqref{eq:int_3} holds an equality if and only if Chebishev's inequality (Lemma~\ref{lem:cheb}) applied to the sequences $\left(d_i^*\right)_{i=2}^{d_1}$ and $\left(\nabla_i\right)_{i=2}^{d_1}$ in \eqref{eq:twosequences} also holds as an equality, and this happens if and only if at least one of the two sequences is constant. By \eqref{eq:twosequences}, the only possibility is that $\left(d_i^*\right)_{i=2}^{d_1}$ is constant. Therefore $\left(d_i^*\right)_{i=2}^{d_1}=\left(d_2^*\right)_{i=2}^{d_1}$. Since $m=d_1+1$, from the identity $\sum_{i=1}d_i^*=\sum_{i=1}d_i=2m$, we deduce 
$$d_2^*(d_1-1)=(d_1+1)$$ 
which implies $d_2^*>1$ and $d_2^*\leq 3$ because the maximum of $\frac{d_1+1}{d_1-1}$ for $d_1\geq 2$ is 3. Hence $1<d_2^*\leq 3$ and thus $d_2^*=3$ by Lemma \ref{lem:twoproperties}.\eqref{com:iii}. Therefore, $3=d_2^*=d_1+1$ and $m=d_1+1=3$. This proves \eqref{com:a_int}. Still by Lemma~\ref{lem:twoproperties}.\eqref{com:iii} we conclude that $d$ is the degree sequence of $K_3\oplus K_{n-3}$. Thus $\Sigma_f(d)=3f(2)$. By the first part of the theorem the maximum of $\Sigma_f$ over $\G(n,3)$ is $3f(1)+f(3)$ because $QS(n,m)$ is a maximizer. Hence $K_3\oplus K_{n-3}$ is also a maximizer if and only if $3f(1)+f(3)=3f(2)$. This proves \eqref{com:b_int} and ends the proof. 
\end{proof}
\section{Some functions belonging to $\F$}
Class $\F$ contains besides the elementary powers also all powers with real exponent $\beta$ such that $\beta\geq 2$. We state this fact in the next theorem whose proof is postponed to the appendix because it is somewhat technical. 
\begin{theorem}\label{thm:powersum}
If $\beta$ is a real number such that $\beta\geq 2$, then $p_\beta\in \F$.   
\end{theorem}
Class $\F$ is closed under taking linear combinations of the form $a_1f_1+a_2f_2+\ldots a_lf_l$ where $l$ is a positive integer, $f_1,\ldots,f_l\in \F$, and $a_1,\ldots,a_l$ are non-negative real numbers such that $\sum_i^la_i^2>0$. This is because each of the defining properties of $\F$ is satisfied by $a_1f_1+a_2f_2+\ldots a_lf_l$ whenever it is satisfied by each of the $f_i$'s. We can call any such a combination a \emph{strict conical combination}. As shown by the next theorem and by Lemma~\ref{lem:hatg}, for some particular choices of the coefficients, $\F$ even contains strictly conical combinations with countably many $f_i$'s. 
\begin{theorem}\label{thm:analytic}
For every real number $a$ such that $a>1$, and for every fixed positive real number $\lambda$, class $\F$ contains the function $f:\R_+\rightarrow \R$ defined by $f(x)=a^{\lambda x}-1$. 
\end{theorem} 
\begin{proof}
Since $f(x)=e^{\ln{a}\lambda x},\quad \forall x\in \R$, it suffices to show the theorem for the function $\widetilde{f}=\exp_\mu-1$. Let $\widetilde{f}:\R\rightarrow \R$ be defined by $x\mapsto e^{\mu x}-1$ with $\mu$ an arbitrary but fixed positive real number. Since $\widetilde{f}$ is centered and strictly convex on $\R_+$, it suffices to show that $\widetilde{f}$ satisfies \eqref{eq:forconvex}. Since the exponential function is analytic with Taylor's series of infinite radius, it coincides with this series centered at zero for any real number $x$. Hence $\widetilde{f}(x)=\sum_{p=0}^\infty a_px^p$ for every $x\in \R$ where $a_0=0$ and $a_p>0$ for every $p\geq 1$. This fact is the only property of the coefficients we need to prove the theorem. Since \eqref{eq:forconvex} holds for the elementary power functions, we can write
\[
\begin{split}
2\widetilde{f}(k)&=2a_0+ \sum_{p=1}^\infty a_1 (2k^p)=\sum_{p=1}^\infty a_1 (2k^p)\\
&\leq \sum_{p=1}^\infty a_p k\left[k^p-(k-1)^p+1\right]\\
&=k\left\{\sum_{p=1}^\infty a_p\left(k^p-(k-1)^p+1\right)\right\}.\\
\end{split}
\]
Since the identity of $\R$, $x\mapsto x$, and the function defined on $\R$ by the map $x\mapsto x-1$  coincide with their Taylor series centered at 0 both with infinite radius, it follows that the function $h$ defined on $\R$ by the map $x\mapsto x(\widetilde{f}(x)-\widetilde{f}(x-1)+1)$ has the same property because the class of analytic functions is closed under product and composition. Therefore the series on the rightmost side of the above inequality is equal to the value of $h$ in $k$ which is precisely $k(\widetilde{f}(k)-\widetilde{f}(k-1)+1)$ and we conclude that
$$2\widetilde{f}(k)\leq k(\widetilde{f}(k)-\widetilde{f}(k-1)+1)$$
which completes the proof.
\end{proof}
Using the following lemma we can prove unique maximization for $\Sigma_{\exp_\lambda-1}$ and $p_\beta$ for all positive real numbers $\lambda$ and $\beta$ with $\beta>2$. 
\begin{lemma}\label{lem:mine_inequality2}
	If $\beta\geq 2$ then $3^{\beta-1}+1-2^\beta\geq 0$. Furthermore, if $\beta>2$ the inequality holds as a strict inequality. 
\end{lemma} 
\begin{proof}
	Define $\phi:\R_+\rightarrow \R$ by $\beta \mapsto 3^{\beta-1}+1-2^\beta$. The derivative of $\phi$ is  $3^{\beta-1}\ln{3}-2^\beta\ln{2}$. This derivative is non-negative if and only if $\beta\geq \xi$ where 
	$$\xi=1+\frac{\ln\left({\ln{4}/\ln{3}}\right)}{\ln{3/2}}$$ 
	so that $\xi$ is the absolute minimum point of $\phi$ and $\xi<2$. It follows that $\phi$ is strictly increasing for $\beta\geq 2$ and since $\phi(2)=0$, we conclude that the stated inequality holds true for every $\beta\geq 2$.     
\end{proof}
\begin{corollary}\label{cor:expandpunique}
If $f=\exp_\lambda-1$ or $f=p_\beta$ for any positive real number $\lambda$ and any real number $\beta$ such that $\beta>2$, then $\Sigma_f$ is uniquely maximized by $QS(m,n)$ over $\G(n,m)$ for any two integers $m$ and $n$ such that $0\leq m\leq n-1$.
\end{corollary}
\begin{proof}
By Theorem~\ref{thm:main}.\eqref{com:ii_inth}, it suffices to check that $f(3)>3(f(2)-f(1))$. If $f=p_\beta$, then the latter inequality is equivalent to the inequality $3^{\beta-1}+1-2^\beta> 0$ which holds true for every $\beta>2$ by Lemma~\ref{lem:mine_inequality2}. Since the latter inequality holds in particular when $\beta\geq 3$ is an integer number, the same reasoning in the proof of Theorem~\ref{thm:analytic} yields $f(3)>3(f(2)-f(1))$ when $f=\exp_\lambda-1$.
\end{proof}
In general, $p_\beta\not\in\F$ when $\beta$ is such that $1<\beta<2$. This is discussed in the next theorem whose proof which is mainly ``calculus'' is postponed to the appendix.
\begin{theorem}\label{ex:1}
Consider the class of instances $\G(m(q)+1,m(q))$ where $m(q)={q\choose 2}$. For every $q\geq 4$, there exists an $\epsilon:=\epsilon(q)$ such that the value of $\Sigma_{1+\epsilon}$ in the degree sequence of $QS(m(q)+1,m(q))$ is strictly less than the value of $\Sigma_{1+\epsilon}$ in the degree sequence of $K_q\oplus\overline{K}_{m(q)+1-q}$. 
\end{theorem}
It follows from the proof of the theorem that $\epsilon(q)\rightarrow 0$ as $q\rightarrow \infty$. In fact, we could not find examples where $\epsilon$ is close to 1 and not even examples where $\epsilon$ is close to $\frac{1}{2}$. In conclusion, the case $1<\beta<2$ is not well understood and deserves further work. 

The following construction will be useful in minimizing $\Sigma_g$ when $-g$ satisfies \eqref{eq:Nstrict}. This happens, for example, when $g$ is the restriction to $\N$ of a strictly concave function. 
\begin{lemma}\label{lem:hatg}
If $g=1-\exp_{-\lambda}$ for some positive real number $\lambda$ or $g=p_\alpha$ for some real number $\alpha$ such that $0<\alpha<1$, then, for all integers $\nu\geq 3$, there exists a function $\hat{g}_\nu:\N\rightarrow \N$ which belongs to $\F$ and such that  
\[
\hat{g}_\nu(k)=	g(\nu)-g(\nu-k),\quad k=0,\ldots,\nu.
\]	
Moreover, $\Sigma_g$ is uniquely maximized by $QS(m,n)$ over $\G(n,m)$ for any two integers $m$ and $n$ such that $0\leq m\leq n-1$.
\end{lemma}
\begin{proof} It suffices to prove that the function $\widetilde{g}_\nu:\{0,1,2\ldots,\nu\}\rightarrow \N$, defined by $\widetilde{g}_\nu(k)=g(\nu)-g(\nu-k)$ satisfies both \eqref{eq:Nstrict} and \eqref{eq:forconvex} for $k\in \{0,1,2\ldots,\nu\}$ and that $\widetilde{g}_\nu$ can be extended to a function defined on the natural numbers belonging to $\F$. It is easy to find one of these extensions: by virtue of Theorem~\ref{thm:analytic}, if $M$ is a sufficiently large positive integer, then  
\[
\hat{g}_\nu(k)=	\begin{cases}
g(\nu)-g(\nu-k) & k=0,\ldots,\nu\\
e^{Mg(\nu)k}-1 & k\geq\nu+1\\
\end{cases}
\]
extends $\widetilde{g}_\nu$ over $\N$, and satisfies \eqref{eq:Nstrict} and \eqref{eq:forconvex} for $k\geq \nu$. Note that this extension does not use the special form of $g$. Since $g$ is strictly concave for any choice of $g$, it follows that $\widetilde{g}_\nu$ satisfies \eqref{eq:Nstrict} over its domain. Hence  $\hat{g}_\nu$ satisfies \eqref{eq:Nstrict} over $\N$. Also, $\hat{g}_\nu$ is centered. So it remains to check that \eqref{eq:forconvex} holds for $k\leq \nu$. Note that if $\nu\leq 3$, then $\hat{g}_\nu$ trivially satisfies \eqref{eq:Nstrict}. Suppose first that $g=1-\exp_{-\lambda}$. In this case 
$$\widetilde{g}_\nu(k)=e^{-\lambda\nu}\left(e^{\lambda k}-1\right).$$
Therefore $\widetilde{g}_\nu$ satisfies \eqref{eq:forconvex} for $k\leq \nu$ because, in light of Theorem~\ref{thm:analytic}, so does any positive multiple of $\exp_\lambda-1$. Now suppose that $g=p_\alpha$ for some $\alpha\in (0,1)$. In this case
$$\widetilde{g}_\nu(k)=\nu^\alpha\left\{1-\left(1-\frac{k}{\nu}\right)^\alpha\right\}.$$ 
The Taylor series of $(1-x)^\alpha$, for $\alpha\in (0,1)$  is (absolutely) convergent (to $(1-x)^\alpha$) for any $x$ such that $|x|<1$ and it is of the form $1-\sum_{\ell\geq 1}b_\ell x^\ell$ where $b_\ell\geq 0,\, \forall\ell\in\N$. Therefore,
$$\widetilde{g}_\nu(k)=\nu^\alpha\left(\sum_{\ell\geq 1}b_\ell \left(\frac{k}{\nu}\right)^\ell\right)=\sum_{\ell\geq 1}c_\ell k^\ell=\sum_{\ell\geq 1}c_\ell p_\ell(k)$$
and we conclude that $\widetilde{g}_\nu$ satisfies \eqref{eq:forconvex} for $k\leq \nu$ because the linear term $p_1$ satisfies \eqref{eq:forconvex} as an equality for every $k\in \N$ and, in light of Theorem~\ref{thm:powersum}, $p_\ell$ satisfies \eqref{eq:forconvex} for $k\leq \nu$ for every integer $\ell\geq 2$. It remains to prove the last part of the lemma. In light of Theorem~\ref{thm:main}, it suffices to show that $\hat{g}_\nu(3)>3(\hat{g}_\nu(2)-\hat{g}_\nu(1))$. Since in both cases $\hat{g}_\nu(3)-3(\hat{g}_\nu(2)-\hat{g}_\nu(1))$ can be expanded as a convergent series with non-negative coefficients of the form $\sum_{\ell\geq 1}b_\ell\left\{p_\ell(3)-3(p_\ell(2)-p_\ell(1))\right\}$, the inequality follows by Corollary~\ref{cor:expandpunique}.
\end{proof}

\section{Functions uniquely minimized by quasi-complete graphs}
In \cite{IsmaStef}, Ismailescu and Stefanica proved that for every $\alpha\in (0,\frac{1}{2}]$, and every $m\leq {n\choose 2}$,  $\Sigma_\alpha$ is uniquely minimized by the quasi-complete graph $QK(n,m)$ over $\G(n,m)$. One of the main arguments in their proof is that if $m\leq {n-1\choose 2}$, then every graph minimizer of $\Sigma_\alpha$ has an isolated vertex. A closer look at the proof of this optimality condition reveals that it is implied by two facts: $-p_\alpha$ satisfies \eqref{eq:Nstrict} and $p_\alpha(k)\geq 2(p_\alpha(k+1)-p_\alpha(k))$. By exploiting the very same idea, a slight abstraction of the latter inequality yields an extension of the optimality condition to a larger class of functions.         
\begin{definition}\label{def:conc}
	The class $\mathcal{G}$ consists of centered, strictly increasing functions $g$ such that $-g$ satisfies \eqref{eq:Nstrict} and $g$ satisfies 
	\begin{equation}\label{eq:forconcave}
		g(k)\geq 2k\left(g(k+1)-g(k)\right),\quad k\in\N.
	\end{equation} 
\end{definition} 
As with class $\F$, and for the same reasons, class $\Gc$ is closed under taking strict conical combinations. The following result provides examples of members of $\Gc$ and at the same time shows a crucial property of these members for proving the main theorem of this section.
\begin{theorem}\label{thm:counterpart}
If $g=1-\exp_{-\lambda}$ for some positive real number $\lambda\geq\ln{2}$ or $g=p_\alpha$ for some real number $\alpha$ such that $0<\alpha\leq\frac{1}{2}$, then $g$ belongs to $\Gc$ and $\hat{g}_\nu\in \F$, for every $\nu\geq 3$.
\end{theorem}
\begin{proof}
The statement $\hat{g}_\nu\in \F$, for every $\nu\geq 3$ follows by specializing Lemma~\ref{lem:hatg}. As in \cite{IsmaStef}, an application of the Mean Value Theorem, shows that $p_\alpha\in \Gc$ for any real number $\alpha$ such that $0<\alpha\leq\frac{1}{2}$. To prove that $1-\exp_{-\lambda}\in \Gc$ for $\lambda\geq \ln2$, consider the function $h$ on $\R_+$ defined by the map $x\mapsto 1-e^{-\lambda x}-2x(e^{-\lambda x}-e^{-\lambda (x+1)})$. It suffices to show that $h\geq 0$ for $\lambda\geq \ln2$ since this implies $h(k)\geq 0,\,k\in \N$, namely, \eqref{eq:forconcave}. Since $1-e^{-\lambda x}\geq \lambda xe^{-\lambda x}$, it follows that $h(x)\geq xe^{-\lambda x}[2e^{-\lambda}+\lambda-2]$. Since $2e^{-\lambda}+\lambda-2\geq 0$ for $\lambda\geq \ln2$, we conclude that $h$ is non-negative on $\R_+$.  
\end{proof}  
\begin{lemma}\label{lem: opt_cond}
If $m\leq {n-1\choose 2}$ and $g\in \Gc$, then every minimizer of $\Sigma_g$ over $\G(n,m)$ has an isolated vertex. 
\end{lemma}
\begin{proof} By Theorem~\ref{thm:opt_1} every minimizer of $\Sigma_g$ is a threshold graph. Let $G$ be a threshold graph with degree sequence $d$. Assume without loss of generality that $d_1\geq d_2\ldots\geq d_n$. It suffices to show that if $G$ has positive minimum degree $d_n$, then $G$ is not a minimizer of $\Sigma_g$. For simplicity let $d_n=k$. By Lemma~\ref{lem:twoproperties}.\eqref{com:i}, every neighbor of vertex $n$ has degree $n-1$. Let $G'$ be the subgraph of $G$ induced by $\{k+1,\ldots,n-1\}$ and let $m'$ be the number of its edges. We claim that $m'\leq {n-k-1\choose 2}-k$. For if not $m'> {n-k-1\choose 2}-k$ and 
	$$m=m'+\left\{k(n-1)-{k\choose 2}\right\}>{n-k-1\choose 2}+k(n-1)-{k+1\choose 2}\geq {n-1\choose 2}$$ 
which contradicts $m\leq {n-1\choose 2}$. Let $\hat{G}$ arise from $G$ by removing all the edges of $G$ incident in $n$ and adding an arbitrary set of $k$ edges to the graph induced by $\{k+1,\ldots,n-1\}$. This is possible because $m'+k\leq {n-k-1\choose 2}$. Let $\hat{d}$ be the degree sequences of $\hat{G}$. Thus 
\[
\hat{d}_i=\begin{cases}
n-2 & i=1,\ldots,k\\
d_i+\epsilon_i & i=k+1,\ldots, n-1\\
0 & i=n, 
\end{cases}
\]
where, for $i=k+1,\ldots,n-1$, $\epsilon_i$ is the change in the degree of vertex $i$ incurred by adding the $k$ edges to the subgraph induced by $\{k+1,\ldots,n-1\}$. Therefore $0\leq \epsilon_i\leq k$ and $\sum_{i=k+1}^{n-1}\epsilon_i=2k$.
We simply check that $g\in \Gc$ implies  $\Sigma_g(d)-\Sigma_g(\hat{d})>0$ so that $G$ is not a minimizer of $\Sigma_g$. One has 
$$\Sigma_g(d)-\Sigma_g(\hat{d})=k\left[g(n-1)-g(n-2)\right]+\left[g(k)-g(0)\right]-\sum_{i=k+1}^{n-1}\left[g(d_i+\epsilon_i)-g(d_i)\right].$$
Since $-g$ satisfies \eqref{eq:Nstrict}, by Lemma~\ref{lem:slope}, it follows that  
$$g(d_i+\epsilon_i)-g(d_i)\leq \epsilon_i\left[g(d_i+1)-g(d_i)\right]\leq \epsilon_i\left[g(k+1)-g(k)\right]$$ 
because $k\leq d_i$ for $i=k+1,\ldots,n-1$. Hence 
$$\sum_{i=k+1}^{n-1}\left[g(d_i+\epsilon_i)-g(d_i)\right]\leq 2k\left[g(k+1)-g(k)\right].$$
Therefore, after recalling that $g(0)=0$, because $g$ is centered, one has 
$$\Sigma_g(d)-\Sigma_g(\hat{d})\geq k\left[g(n-1)-g(n-2)\right]+g(k)-2k\left[g(k+1)-g(k)\right].$$
Since for every $k\in \N$ it holds that $k\left[g(n-1)-g(n-2)\right]>0$ (because $g$ is strictly increasing) and $g(k)-2k\left[g(k+1)-g(k)\right]$ (because $g$ satisfies \eqref{eq:forconcave}) we conclude that $\Sigma_g(d)-\Sigma_g(\hat{d})>0$.
\end{proof}

\begin{theorem}\label{thm:mainconcave}
	If $g=1-\exp_{-\lambda}$ for some real number $\lambda\geq\ln{2}$ or $g=p_\alpha$ for some real number $\alpha$ such that $0<\alpha\leq\frac{1}{2}$, then $QK(n,m)$ uniquely minimizes $\Sigma_g$ over $\G(n,m)$ for any two integers $m$ and $n$ such that $n>0$ and $0\leq m\leq {n\choose 2}$.   
\end{theorem}
\begin{proof}
Let $\ell\leq n$ be the unique positive integer such that 
\begin{equation}\label{eq:int_last}
{\ell-1\choose 2}< m\leq {\ell \choose 2}.
\end{equation}
Since $g\in \Gc$, every graph minimizer of $\Sigma_g$ has at least $n-\ell$ isolated vertices: if $\ell=n$, then this fact is trivial otherwise, if $\ell<n$, then it is granted by Lemma~\ref{lem: opt_cond}. Since isolated vertices do not contribute to $\Sigma_g$, the minimum of $\Sigma_g$ over $\G(n,m)$ is the same as the minimum of $\Sigma_g$ over $\G(\ell,m)$ and any minimizer in the latter set can be uniquely extended to a minimizer of the former by adding isolated vertices. Note that, with some abuse of notation, we use the same symbol $\Sigma_g$ for both the original objective function and for the objective function of the minimization over $\G(\ell,m)$.
Let $\mathbb{D}(\ell,m)$ be the set of the degree sequences of the graphs of $\G(\ell,m)$ and let $\overline{m}={\ell \choose 2}-m$. The map $G\mapsto \overline{G}$ is a bijection between $\G(\ell,m)$ and $\G(\ell,\overline{m})$ and the map $d\mapsto \overline{d}$, where $\overline{d}=(\ell-1-d_1\ldots,\ell-1-d_\ell)$ is the degree sequence of $\overline{G}$, is a bijection  between $\mathbb{D}(\ell,m)$ and $\mathbb{D}(\ell,\overline{m})$. Let $\nu=\ell-1$ and consider the function $\hat{g}_\nu$ defined in Lemma~\ref{lem:hatg}. Since $\hat{g}_\nu(\nu)=g(\nu)$, for every $k\in \{0,\ldots,\nu\}$ one has the identity
$$g(\nu)=g(k)+\hat{g}_\nu(\nu-k)=\hat{g}_\nu(\nu)$$
and thus, for any $d\in \mathbb{D}(\ell,m)$,
$$\sum_ig(d_i)=\ell g(\nu)-\sum_i\hat{g}_\nu(\nu-d_i)$$ 
which implies 
$$\min_{\mathbb{D}(\ell,m)}\Sigma_g=\ell g(\nu)-\max_{\mathbb{D}(\ell,\overline{m})}\Sigma_{\hat{g}_\nu}.$$ 
Therefore, any minimizer of $\Sigma_g$ over $\G(\ell,m)$ is the complement of a maximizer of $\Sigma_{\hat{g}_\nu}$ over $\G(\ell,\overline{m})$. Since $\overline{m}\leq \ell-1$ (because of \eqref{eq:int_last}) and $\hat{g}_\nu\in \F$ (because of Theorem~\ref{thm:counterpart}), it follows that $\Sigma_{\hat{g}_\nu}$ is uniquely maximized by $QS(\ell,\overline{m})$ (because of the last part of Lemma~\ref{lem:hatg}). We conclude that $\Sigma_g$ is uniquely minimized by $QK(\ell,m)$ and hence that $\Sigma_g$ is uniquely minimized by $QK(n,m)$.    
\end{proof}
Theorem \ref{thm:mainconcave} can be stated more generally. In fact, the class of functions qualifying for the thesis are precisely those functions $g$ such that $g\in \Gc$ and $\hat{g}_\nu\in \F$ under the condition that $\hat{g}_\nu(3)>3(\hat{g}_\nu(2)-\hat{g}_\nu(1))$, which grants uniqueness in maximizing the complement.
\begin{theorem}\label{thm:mainconcavegen}
	Let $g\in \Gc$. If $\hat{g}_\nu\in \F$ for some $\nu\geq 3$ and $\hat{g}_\nu(3)>3(\hat{g}_\nu(2)-\hat{g}_\nu(1))$, then $QK(n,m)$ uniquely minimizes $\Sigma_g$ over $\G(n,m)$ for any two integers $m$ and $n$ such that $0<n\leq \nu-1$ and $0\leq m\leq {n\choose 2}$.   
\end{theorem}
Using this theorem, it is possible to construct other functions $g$ in such a way that $\Sigma_g$ is uniquely minimized by a quasi-complete graph. For example, if $g:\R_+\rightarrow \R$ is defined by $x\mapsto \frac{x}{x+1}$, then it is easy to see that $g\in \Gc$. Furthermore,  
$$\hat{g}_\nu(k)=\frac{1}{\nu+1}\left\{\left(1-\frac{k}{\nu+1}\right)^{-1}-1\right\}=\frac{1}{\nu+1}\sum_{\ell\geq 1}\left(\frac{k}{\nu+1}\right)^{\ell}=\sum_{\ell\geq 1}a_\ell p_\ell(k),$$
where we set $a_\ell=(\nu+1)^{-(\ell+1)}$. Since $a_\ell\geq 0$, $p_\ell\in \F$ for every $\ell\geq 1$ (for $\ell=1$ this is trivial while for $\ell\geq 2$ it follows by Theorem~\ref{thm:powersum}), and $p_\ell(3)>3(p_\ell(2)-p_\ell(1))$ for every $\ell\geq 3$ (Lemma~\ref{lem:mine_inequality2}), we conclude that $\hat{g}_\nu\in \F$ and that $\hat{g}_\nu(3)>3(\hat{g}_\nu(2)-\hat{g}_\nu(1))$. Therefore, $\Sigma_g$ is uniquely minimized by quasi-complete graphs by Theorem~\ref{thm:mainconcavegen}.

\section*{Appendix: Proofs of Theorem~\ref{thm:powersum} and Theorem~\ref{ex:1}}
Here we give proofs of Theorems~\ref{thm:powersum} and~\ref{ex:1}. 
\paragraph{Proof of Theorem~\ref{thm:powersum}}
First, we need the following lemma.
\begin{lemma}\label{lem:mine_inequality}
Let $x$ and $\beta$ be real numbers. If $x\geq 3$ and $\beta\geq 2$, then
\begin{equation}\label{eq:int_a1}
	\left(1-\frac{1}{x}\right)^{\beta-1}+\left(\frac{\beta-1}{\beta}\right)\frac{2}{x}\leq 1.
\end{equation} 
\end{lemma} 
\begin{proof}
Let $D=\{x\in \R |\ x\geq 3 \}$. To prove that \eqref{eq:int_a1} holds true for every $x\in D$ and every $\beta\geq 2$, we distinguish three cases: $\beta\geq 4$, called \textbf{Case 1}, $3\leq \beta< 4$, called \textbf{Case 2} and, finally, $2\leq \beta< 3$, called \textbf{Case 3}.
\mybreak
\textbf{Case 1}. Since $\beta-1\geq 3$,
$$\left(1-\frac{1}{x}\right)^{\beta-1}+\left(\frac{\beta-1}{\beta}\right)\frac{2}{x}\leq \left(1-\frac{1}{x}\right)^3+\frac{2}{x}\leq 1,\quad \forall x\in D$$
as it is readily checked by expanding the cube. Hence \eqref{eq:int_a1} holds in \textbf{Case 1}.
\mybreak
\textbf{Case 2}. Since $3\leq\beta< 4$ implies $\beta-1\geq 2$ and $\frac{\beta-1}{\beta}\leq \frac{3}{4}$, one has
$$\left(1-\frac{1}{x}\right)^{\beta-1}+\left(\frac{\beta-1}{\beta}\right)\frac{2}{x}\leq \left(1-\frac{1}{x}\right)^2+\frac{3}{2x}\leq 1,\quad \forall x\in D$$
as it is readily checked by expanding the square. Hence \eqref{eq:int_a1} also holds in \textbf{Case 2}.
\mybreak
\textbf{Case 3}. We use the following inequality (\cite{mitri}, Theorem 5 p.~35) where $a$ and $y$ are positive real numbers such that $a<3$ and $y<1$ 
$$(1-y)^a<1-ay+\frac{1}{2}a(a-1)y^2.$$
The above inequality yields
$$\left(1-\frac{1}{x}\right)^{\beta-1}\leq 1-(\beta-1)\frac{1}{x}+\frac{(\beta-1)(\beta-2)}{2x^2}.$$ 
Hence
$$\left(1-\frac{1}{x}\right)^{\beta-1}+\left(\frac{\beta-1}{\beta}\right)\frac{2}{x}\leq 1-(\beta-1)\frac{1}{x}+\frac{(\beta-1)(\beta-2)}{2x^2}+\left(\frac{\beta-1}{\beta}\right)\frac{2}{x}.$$ 
Since 
$$\frac{2}{x}\left(\frac{\beta-1}{\beta}\right)+\frac{(\beta-1)(\beta-2)}{2x^2}\leq (\beta-1)\frac{1}{x},\quad \forall x\in D,$$
it follows that \eqref{eq:int_a1} holds in \textbf{Case 3} as well and this completes the proof. 
\end{proof}
Now we can prove the theorem. Let $D$ be as in the lemma. Since $p_\beta$ is a centered and strictly convex function on $\R_+$, we only need to check that $p_\beta$ satisfies \eqref{eq:forconvex}. So we only need to check that for every integer $k$ such that $k\geq 3$ it holds that  
	$$(k-2)k^\beta-k\left((k-1)^\beta-1\right)\geq 0$$
or, equivalently, that
$$k^\beta-(k-1)^\beta-2k^{\beta-1}+1\geq 0.$$
The function $\psi:\R_+\rightarrow \R$ defined by
$$\psi(x)=x^\beta-(x-1)^\beta-2x^{\beta-1}+1$$ 
is non-decreasing in $x$ for all $x\in D$. To see this, take the derivative $\psi'$ of $\psi$ and impose that $\psi'\geq 0$, i.e., 
$$\beta x^{\beta-1}-\beta(x-1)^{\beta-1}-2(\beta-1)x^{\beta-2}\geq 0.$$
Since $x^{\beta-1}>0$ for $x>0$, on dividing both sides of the above inequality by $x^{\beta-1}$ yields inequality \eqref{eq:int_a1} which, by Lemma~\ref{lem:mine_inequality}, holds true for every $x\in D$ and every $\beta\geq 2$. Therefore, $\psi$ is increasing in $x$ over $D$ and  
$$\psi(x)\geq \psi(3),\quad \forall x\geq 3.$$
Since by Lemma~\ref{lem:mine_inequality2}
$$\psi(3)=3^\beta-2^\beta+1-2\cdot 3^{\beta-1}=3^{\beta-1}-2^\beta+1\geq 0$$  
we conclude that $\psi(x)\geq 0$ for all real numbers $x\geq 3$ and hence for all integers $k$ such that $k\geq 3$. In particular, \eqref{eq:forconvex} is satisfied by $p_\beta$ for every $\beta\geq 2$ which is precisely what we wanted to prove.
\paragraph{Proof of Theorem~\ref{ex:1}} Observe that $QS(m(q)+1,m(q))\cong K_{1,m(q)}$. Since the value of $\Sigma_{1+\epsilon}$ in the degree sequence of $K_{1,m(q)}$ is $m(q)+\left(m(q)\right)^{1+\epsilon}$ while the value of $\Sigma_{1+\epsilon}$ in the degree sequence of $K_q\oplus\overline{K}_{m(q)+1-q}$ is $q\left(q-1\right)^{1+\epsilon}$, we must show that there exists some $\epsilon$ which satisfies 
$$
m(q)+\left(m(q)\right)^{1+\epsilon}<q\left(q-1\right)^{1+\epsilon}
$$  
which, after simplifying,  reduces to the inequality
\begin{equation}\label{eq:int_exa}
	\left(q-1\right)^{\epsilon}\left[2-\left(\frac{q}{2}\right)^\epsilon\right]>1.
\end{equation}
Let $h: [0,1]\rightarrow \R$ be the function defined by 
$$h(\epsilon)=\left(q-1\right)^{\epsilon}\left[2-\left(\frac{q}{2}\right)^\epsilon\right]-1.$$
Clearly $h(0)=0$ and $h(1)<0$ for any fixed $q\geq 4$. 
The first and the second derivatives of $h$ are, respectively, 
$$h'(\epsilon)=2\ln(q-1)(q-1)^\epsilon-\left(\ln(\frac{q}{2})+\ln(q-1)\right)(q-1)^\epsilon\left(\frac{q}{2}\right)^\epsilon$$
and
$$h''(\epsilon)=2\left[\ln(q-1)\right]^2(q-1)^\epsilon-\left(\ln(\frac{q}{2})+\ln(q-1)\right)^2(q-1)^\epsilon\left(\frac{q}{2}\right)^\epsilon.$$
One has $h''(\epsilon)<0$ if and only if 
$$2<\left(\frac{q}{2}\right)^\epsilon\left\{1+\frac{\ln{\frac{q}{2}}}{\ln{q-1}}\right\}^2$$
which is satisfied by all $\epsilon\in (0,1]$ because 
$$\left(\frac{q}{2}\right)^\epsilon>1,\quad \quad \frac{\ln{2}}{\ln{3}}\leq \frac{\ln{q/2}}{\ln{q-1}}< 1,\quad \text{\rm and}\quad 2\left(\frac{\ln{3}}{\ln{6}}\right)^2<1\quad \forall q\geq 4.$$
Hence $h$ s strictly concave on $(0,1]$ and, consequently, $h$ has at most two roots, one which is zero, and the other, if any, is denoted by $\xi_0$. The derivative of $h$ vanishes if and only if 
$$
2=\left\{1+\frac{\ln{\frac{q}{2}}}{\ln{q-1}}\right\}\left(\frac{q}{2}\right)^\epsilon.
$$
Since  
$$\frac{\ln{6}}{\ln{3}}\left(\frac{q}{2}\right)^\epsilon\leq \left\{1+\frac{\ln{\frac{q}{2}}}{\ln{q-1}}\right\}\left(\frac{q}{2}\right)^\epsilon< 2\left(\frac{q}{2}\right)^\epsilon,\quad \forall q\geq 4$$
the (unique) root $\xi$ of $h'$ satisfies the inequalities   
$$1<\left(\frac{q}{2}\right)^\xi\leq \frac{\ln{9}}{\ln{6}}$$
and, consequently,
\begin{equation}\label{eq:int_exa_1}
	0<\xi<\ln\left(\frac{\ln{9}}{\ln{6}}\right)/\ln\left(q/2\right).
\end{equation}
Thus $h$ has an absolute maximum point $\xi$ which is located in the interval defined in \eqref{eq:int_exa_1}. Since $\xi$ is an absolute maximum point, $0=h(0)<h(\xi)$ holds. So $h$ has two roots and $h$ is positive between $0$ and $\xi_0$. We conclude that inequality \eqref{eq:int_exa} is satisfied by any $\epsilon$ such that $0<\epsilon<\xi_0$ and hence, for any such $\epsilon$, $K_{1,m(q)}$ is not a maximizer of $\Sigma_{1+\epsilon}$ over $\G(m(q)+1,m(q))$.

\end{document}